\documentclass[11pt]{amsart}
\usepackage{amsfonts,latexsym, color}
\usepackage{amsmath}
\usepackage{amssymb}
\usepackage[colorlinks=true, bookmarks=true, pdfstartview=FitH, pagebackref=true, linktocpage=true]{hyperref}

\allowdisplaybreaks

\numberwithin{equation}{section}

\newtheorem{thm}{Theorem}[section]
\newtheorem{lem}[thm]{Lemma}
\newtheorem{cor}[thm]{Corollary}
\newtheorem{prop}[thm]{Proposition}
\newtheorem{Def}[thm]{Definition}

\newcommand{\A}{\mathcal{A}}
\newcommand{\B}{\mathcal{B}}

\newcommand{\C}{\mathbb{C}}

\newcommand{\N}{\mathbb{N}}

\newcommand{\R}{\mathbb{R}}

\newcommand\cc{\mathcal{C}}

\def \p{\partial}

\begin{document}

	\title[ Finsler $N$-Liouville Equation ]{Quantization of blow-up masses for the Finsler $N$-Liouville equation}
	
	\author[X.\ Huang]{Xia Huang}
	\author[Y.\ Li]{ Yuan Li}
\author[D.\ Ye]{ Dong Ye}
	\author[F.\ Zhou]{Feng Zhou}

	\address{School of Mathematical Sciences, East China Normal
Universit, Key Laboratory of MEA (Ministry of Education) and Shanghai Key Laboratory of PMMP, Shanghai, 200241, China}
	\email{ xhuang@cpde.ecnu.edu.cn}

	\address{School of Mathematical Sciences, East China Normal
Universit, Key Laboratory of MEA (Ministry of Education) and Shanghai Key Laboratory of PMMP, Shanghai, 200241, China}
\email{ liyuan5397@163.com}

	\address{School of Mathematical Sciences, East China Normal
Universit, Key Laboratory of MEA (Ministry of Education) and Shanghai Key Laboratory of PMMP, Shanghai, 200241, China}
\email{ dye@math.ecnu.edu.cn}

	\address{School of Mathematical Sciences, East China Normal
Universit, Key Laboratory of MEA (Ministry of Education) and Shanghai Key Laboratory of PMMP, Shanghai, 200241, China\\
NYU-ECNU Institute of Mathematical Sciences at NYU-Shanghai, Shanghai 200120, China}
\email{ fzhou@math.ecnu.edu.cn}


	\begin{abstract}
The quantization results for blow-up phenomena play crucial roles in the analysis of partial differential equations. Here we quantify the blow-up masses to the following Finsler $N$-Liouville equation $$-Q_{N}u_{n}=V_{n}e^{u_{n}}\quad\mbox{in}~ \Omega\subset \mathbb{R}^{N}, N \ge 2.$$
Our study generalizes the classical result of Li-Shafrir \cite{LS1994} for Liouville equation, Wang-Xia's work for anisotropic Liouville equation in $\R^2$ \cite{WX2012}, and Esposito-Lucia's for the $N$-Laplacian case in $\R^N$ ($N \geq 3$) in their recent paper \cite{EL2021}.
	\end{abstract}
	\maketitle

\section{Introduction}\label{sec1}
In this paper, we are interested in the following Finsler $N$-Liouville equation
\begin{align}\label{main eq1}
-Q_{N}u =V e^{u} \quad \mbox{in}\; \Omega,
\end{align}
with a bounded smooth domain $\Omega\subset\mathbb{R}^{N} (N\geq2)$, $V \in C(\Omega)$ and the operator $Q_{N}$ defined by
$$Q_{N}u:= {\rm div}\Big(F^{N-1}(\nabla u)F_{\xi}(\nabla u)\Big) = \sum_{i=1}^{N}\partial_{x_i}\Big(F^{N-1}(\nabla u)F_{\xi_{i}}(\nabla u)\Big).$$
Here $F_\xi = (F_{\xi_i})$, $F_{\xi_i} := \frac{\partial F}{\partial \xi_i}$, and $F$ is a convex function over $\mathbb{R}^N$ satisfying
\begin{itemize}
\item[--] $F \in C^{2}$ in $\mathbb{R}^{N}\backslash\{0\};$
\item[--] $F(t\xi)=|t|F(\xi)$ for any $t\in\mathbb{R}$, $\xi \in\mathbb{R}^{N};$
\item[--] there exists $0< C_1 \leq C_{2}<\infty$ such that $C_{1}|\xi|\leq F(\xi)\leq C_{2}|\xi|$ for any $\xi \in \mathbb{R}^{N}$.
\end{itemize}
In fact, $Q_{N}$ is called the Finsler $N$-Laplacian or anisotropic $N$-Laplace operator, and it has been studied vastly in the literature. Historically, Wulff \cite{W1901} used such operators to study crystal shapes and minimization of anisotropic surface tensions. For more background on anisotropic operators and Finsler geometry,
%\footnote{cite the book of Shen ZM, I cite the book of Yi-Bin Shen and Zongming Shen},
we refer the interested readers to \cite{AT1995, ATW1993, AFTL1997, CS2009, CFV2014, FM1991, SS2016} and references therein.

The most well known example of $F$ is
$$F(\xi)= |\xi|_q := \Big(\sum\limits_{i=1}^{N}|\xi_{i}|^q\Big)^{\frac{1}{q}} \;\; \mbox{ in}\; \mathbb{R}^N$$
with $1<q<\infty$. If moreover $q=2$, $F(\xi)= |\xi|$ is the Euclidean norm and $Q_{N}$ is the classical isotropic $N$-Laplacian, i.e.~${\rm div}\big(|\nabla u|^{N-2}\nabla u\big)$.  Especially, when $N=2$,  we arrive at
%the Laplacian and equation (\ref{main eq1}) becomes
\begin{align}\label{eq2}
-\Delta u=V e^{u} \quad\mbox{in } \Omega \subset \mathbb{R}^{2}.
\end{align}
It's well known that equation (\ref{eq2}) has profound geometric and physical background. In differential geometry, \eqref{eq2} is called the prescribed Gauss curvature equation see \cite{KW1974,N1982,CN1991}. It is also linked to the mean field equation, the Chern-Simons model and so on, see \cite{BDM2015, BLT2011, FT2014} and references therein.

In particular, when $V\equiv1$, Liouville \cite{L1853} studied (\ref{eq2}) and gave the expression of all smooth solutions using holomorphic functions, which means that there is a rich family of solutions over $\C$. Moreover, Chen-Li \cite{CL} classified all solutions of  $-\Delta u = e^u$ in $\R^2$ if $e^{u}\in L^{1}(\mathbb{R}^{2})$, they showed that any solution is given by
$$u_*(x)=\ln\frac{8\lambda^{2}}{(1+\lambda^{2}|x-p_0|^{2})^{2}}, \quad \mbox{with } \lambda > 0, \, p_0\in\mathbb{R}^{2}.$$
Hence for any $p_0 \in \R^2$ and $\lambda > 0$,
$$\int_{\mathbb{R}^{2}}e^{u_*}dx=8\pi.$$

To understand the blow-up phenomena or the compactness of family of solutions of  \eqref{eq2} with various $V$, Brezis-Merle \cite{BM1991} showed a milestone compactness-concentration result:

\medskip
\noindent
{\bf Theorem A.} {\it Assume that $\{u_{n}\}$ is a sequence of solutions to
$$-\Delta u_{n}=V_{n} e^{u_{n}} \quad \mbox{in } \Omega\subset\mathbb{R}^{2},$$
such that
$$0\leq V_{n}(x)\leq M, \;\;\forall\; x\in\Omega; \quad \int_{\Omega}e^{u_{n}}dx\leq M$$
with some constant $M > 0$. Then up to a subsequence, we have one of the following three alternatives:
\begin{itemize}
\item[(i)] $\{u_{n}\}$ is bounded in $L^{\infty}_{loc}(\Omega)$;
\item[(ii)] $u_{n}\rightarrow-\infty$ uniformly on compact subsets of $\Omega$;
\item[(iii)] There exists a finite blow-up set $S=\{a_{1},\cdots,a_{m}\}\subset\Omega$ such that $u_{n}(x)\rightarrow-\infty$ uniformly on any compact subset of  $\, \Omega\backslash S$. Moreover, $V_{n}e^{u_{n}}\rightharpoonup\sum_{1 \le i \le m}\alpha_{i}\delta_{a_{i}}$ in the sense of measure with $\alpha_{i}\geq4\pi$, $1\leq i\leq m$.
\end{itemize}}

\noindent This result plays a crucial role in the study of equations like \eqref{eq2}, for the prescribing Gauss curvature problem or mean field equation: the existence issue, blow-up analysis, asymptotic behavior near isolated singularity, compactness and many other aspects of solutions. It has inspired so many important works, it is just impossible for us to cite some of them, seeing the hundreds of citations.
%, we refer to the representative \cite{BT2017,CL1998,DJLW1999,T2023} and references therein for interested readers. %\footnote{\color{red} Should we choose more  important and representative references}
%see \cite{DLR2018,JZZ2019,JLW2006} and references therein.

Brezis-Merle conjectured (see \cite[Open Problem 4]{BM1991}) that under some minor regularity assumptions on ${\{V_n}\}$, the concentration Dirac mass $\alpha_i$ appearing in alternative $(iii)$ should belong to $8\pi\N$, in other words the multiple of bubble energy classified by Chen-Li. Li-Shafrir \cite{LS1994} gave an affirmative answer to this conjecture, assuming that $V_{n}\rightarrow V$ in $C(\overline{\Omega})$. Furthermore, if $\{V_{n}\}$ are uniformly Lipschitz functions, i.e.
\begin{align}\label{VC}
V_{n}\geq0, \quad V_{n}\rightarrow V \mbox{ in }C(\overline{\Omega}) \quad \mbox{and} \quad\|\nabla V_{n}\|_{L^{\infty}(\Omega)}\leq M,
\end{align}
and if
\begin{align}\label{+1}
\sup\limits_{\partial\Omega}u_{n}-\inf\limits_{\partial\Omega}u_{n}\leq M,
\end{align}
for some positive constant $M$, Li \cite{L1999} proved that $\alpha_i=8\pi$ for each $i$.

\medskip
Recently, the above studies have been generalized to the $N$-Laplacian case, that is, for
\begin{align}\label{main eq3}
-\Delta_{N}u=V e^{u} \quad\mbox{in } \; \Omega \subset \mathbb{R}^{N}, \; N \geq 2.
\end{align}
Firstly, assuming that $V=1$, Esposito \cite{E2018} classified that all solutions to equation (\ref{main eq3}) with $e^u \in L^1(\R^N)$ are as follows:
$$u(x)=\ln\frac{\mathcal{C}_{N}\lambda^{N}}{\big(1+\lambda^{\frac{N}{N-1}}|x-p_0|^{\frac{N}{N-1}}\big)^{N}} \quad\mbox{and}\quad \int_{\mathbb{R}^{N}}e^udx=\mathcal{C}_{N}\omega_{N},$$
where $p_0\in\mathbb{R}^{N}$, $\lambda > 0$, $\mathcal{C}_{N}=N\left(\frac{N^{2}}{N-1}\right)^{N-1}$ and $\omega_{N}$ stands for the volume of the unit Euclidean ball in $\mathbb{R}^{N}$. Moreover, Brezis-Merle's trichotomy result, the corresponding Li-Shafrir or Li type blow-up mass quantification have been extended to \eqref{main eq3} respectively in \cite{AP1997, EL2021, EM2015}.

\medskip
A natural question is to consider more general anisotropic case. To describe precisely the result for anisotropic situation, we need to introduce the Wulff ball and precise also the meaning of solution here. Let $F^{0}$ be the support function of $K:=\{x\in\mathbb{R}^{N}\mid ~F(x)<1\}$, that is
$$F^{0}(x):=\sup_{\xi\in K} \langle x,\xi\rangle,$$
where $\langle \cdot,\cdot\rangle$ denotes the inner product in $\mathbb{R}^{N}$. $F^{0}$ is indeed the dual norm of $F$, and $F^0$ is a convex function with first-order positive homogeneity. For more properties of $F^{0}$, we refer to \cite{WX2012,XG2016} and references therein. We denote by $\mathcal{B}_{r}(p):=\{x\in\mathbb{R}^{N}: F^{0}(x-p)<r\}$ the Wulff ball with respect to $F^0$, of radius $r$ centered at $p \in \mathbb{R}^N$. We write $\mathcal{B}_{r}$ when the center is the origin.
\begin{Def}
We say that $u$ is a weak solution of (\ref{main eq1}), if $u\in W^{1,N}_{loc}(\Omega)$, $Ve^{u}\in L^{1}_{loc}(\Omega)$ and
$$\int_{\Omega}F^{N-1}(\nabla u)F_{\xi}(\nabla u)\cdot\nabla\phi dx=\int_{\Omega}Ve^{u}\phi dx, \quad \forall\; \phi\in W_{c}^{1, N}(\Omega)\cap L^\infty(\Omega).$$
\end{Def}

To the best of our knowledge, there are  few works considering the blow-up analysis of genral anisotropic Liouville equations. Wang-Xia \cite{WX2012} first proved the concentrated compactness property of  solutions to (\ref{main eq1}) when $N=2$, and extended the results of \cite{BM1991, L1999} to the anisotropic Liouville equation in $\R^2$. Later on, Xie-Gong \cite{XG2016} considered Brezis-Merle's trichotomy result for \eqref{main eq1}, and proved

\medskip
\noindent
{\bf Theorem B.}
%\begin{thm}(\cite{XG2016})\label{thm1.2}
{\it Let $\Omega$ be a bounded smooth domain and $u_{n}$ be a sequence of weak solutions to
\begin{align}\label{main eq4}
-Q_{N}u_{n} =V_{n}e^{u_{n}} \quad \mbox{in} \; \Omega\subset\mathbb{R}^{N},
\end{align}
 with
\begin{align}\label{Vn}
0\leq V_{n}\leq M \quad\mbox{and}\quad \int_{\Omega}e^{u_{n}}dx\leq M .
\end{align}
Then up to a subsequence, one of the following alternatives holds true:
\begin{itemize}
\item[(i)] $\{u_{n}\}$ is bounded in $L^{\infty}_{loc}(\Omega)$;
\item[(ii)] $u_{n}\rightarrow-\infty$ uniformly on compact subsets of $\Omega$;
\item[(iii)] There is a nonempty finite set $S=\{a_{1},\cdots,a_m\}\subset\Omega$, and $u_{n}(x)\rightarrow-\infty$ uniformly on any compact subset of $\Omega\backslash S$. In addition, $V_{n}e^{u_{n}}\rightharpoonup\sum_{1\le i \le m}\alpha_{i}\delta_{a_{i}}$ in the sense of measure on $\Omega$ with $\alpha_{i}\geq N^{N}\kappa$ for any $i$, where $\kappa:=|\mathcal{B}_1|$.
\end{itemize}
Moreover, if (\ref{VC}) and (\ref{+1}) are satisfied, then $\alpha_i=\mathcal{C}_{N}\kappa$ in $(iii)$ with
$\mathcal{C}_{N} = N\Big(\frac{N^2}{N-1}\Big)^{N-1}$.}

\medskip
Now we can ask whether the quantification of concentration mass $\alpha_i$ holds true in more general setting of \eqref{main eq4}, analogous to Li-Shafrir's result on Brezis-Merle's conjecture. Our main objective here is to give an affirmative answer.

\begin{thm}\label{thm1}
Let $\{u_{n}\}$ be a sequence of weak solution to (\ref{main eq4}) with $V_{n}\geq0$, $V_{n}\rightarrow V$ in $C(\overline{\Omega})$ and
\begin{align}\label{energy}
\int_{\Omega}e^{u_{n}}dx\leq M < \infty\quad\mbox{ for any } n.
\end{align}
Assume that the alternative $(iii)$ happens for $u_n$, then for any $1\le i \le m$, $\alpha_i \in \mathcal{C}_{N}\kappa\N$.

Furthermore, $\alpha_1 = {C}_{N}\kappa$ if there is $R > 0$ such that $\mathcal{B}_R(a_1) \subset \Omega$ and
\begin{align}\label{th6}
u_{n}(x)+N\ln F^{0}(x-a_1)\leq M', \quad \mbox{ for any } n \in \N, x\in\mathcal{B}_{R}(a_1).
\end{align}
%Then $\alpha_1 = {C}_{N}\kappa$.
\end{thm}

The proof is mainly based on the blow-up analysis, and the recent classification result of Ciraolo-Li \cite{CL2023} for the anisotropic $N$-Liouville equations under the integrability assumption. More precisely, they established that any solution to
\begin{align}
\label{bubble}
-Q_{N}u=e^u \;\;\mbox{in } \mathbb{R}^{N},\quad \int_{\mathbb{R}^{N}}e^{u}dx<\infty,
\end{align}
satisfies
\begin{align}
\label{bubble1}
u(x)=\ln\frac{\mathcal{C}_{N}\lambda^{N}}{\Big(1+\lambda^{\frac{N}{N-1}} F^{0}(x-p)^{\frac{N}{N-1}}\Big)^{N}} \quad\mbox{and}\quad \int_{\mathbb{R}^{N}}e^{u}dx=\mathcal{C}_{N}\kappa,
\end{align}
where $\lambda > 0$ and $p\in\mathbb{R}^{N}$. In the anisotropic case, we say that a function is {\it radial} with respect to $p \in \R^N$ if it depends only on $F^{0}(x-p)$.

\medskip
Our proof consists mainly of three steps. First, by blow-up around a local maximum of the function $u_n$, we are led to a solution of \eqref{bubble}, characterized in \cite{CL2023}. Then we will iterate this blow-up procedure around other eventual local maxima, and obtain each time a contribution of $\mathcal{C}_N \kappa$ to the limiting concentration mass. According to \eqref{energy}, the procedure will stop after a finite number of iterations. The last ingredient consists of showing that there is no mass contribution out of the bubbles, or the so called ``neck region".

Our paper is organized as follows. Section \ref{sec3} is devoted to establish some {\it a priori} estimates such as {\bf$\sup+\inf$} type inequality and Harnack type inequality, which are essential for our analysis. In section \ref{sec4}, we apply the blow up argument to prove the quantization result.

\section{A priori estimates}\label{sec3}
First of all, by regularity theory for anisotropic operators (see \cite{D1983, S1964} or \cite[Proposition 2.2]{CL2023}), we know that any solution to \eqref{main eq1} belongs to $C^{1, \beta}_{loc}(\Omega)$.

Here we expose some crucial {\it a priori} estimates for the bow-up analysis. A first one is the following Harnack type inequality for anisotropic equation, in the spirit of \cite[Theorem 3.1]{KM1992}.

\begin{prop}\label{pro2.1}
Let $p \in \R^N$, $r > 0$ and $u$ be a continuous and weak solution of
$$-Q_{N}u=f\geq0 \quad\mbox{in}\quad \mathcal{B}_{2r}(p)\subset\mathbb{R}^{N}.$$
Then
\begin{align}\label{estimate}
u(p)-\inf_{\mathcal{B}_{2r}(p)}u(x)\geq C\int_{0}^{r}\Big(\int_{\mathcal{B}_{t}(p)}fdx\Big)^{\frac{1}{N-1}}\frac{dt}{t},
\end{align}
where $C$ is a positive constant depending on $N$ and $F$.
%\footnote{\color{red} It is somehow strange that the constant does not depend on $F$, so need to be checked}
\end{prop}

The constant $C$ in \eqref{estimate} is not explicit in general. However, in the class of radial functions, \eqref{estimate} holds with the sharp constant $C=(N\kappa)^{-\frac{1}{N-1}}$ as follows.
\begin{lem}\label{lem2.2}
For $R>0$, let $f\in C(\overline{\mathcal{B}_R}(p))$ be a radial function with respect to $p\in\mathbb{R}^{N}$. Assume that $u\in C(\mathcal{B}_{R}(p))$ satisfies
$$-Q_{N}u\geq f \geq 0 \quad\mbox{in}~ \mathcal{B}_{R}(p)$$
in the weak sense. Then
$$u(p)-\inf_{\mathcal{B}_{R}(p)}u(x)\geq (N\kappa)^{-\frac{1}{N-1}}\int_{0}^{R}\Big(\int_{\mathcal{B}_{t}(p)}fdx\Big)^{\frac{1}{N-1}}\frac{dt}{t}.$$
In particular, for each $0<r<R$, there holds
$$u(p)-\inf_{\mathcal{B}_{R}(p)}u(x)\geq (N\kappa)^{-\frac{1}{N-1}}\Big(\int_{\mathcal{B}_{r}(p)}fdx\Big)^{\frac{1}{N-1}}\ln\frac{R}{r}.$$
\end{lem}

\begin{proof}
Direct calculation shows that the unique radial solution to
\begin{equation*}
-Q_{N}u_{0}=f \;\; \mbox{in} ~ \mathcal{B}_{R}(p),\quad u_{0}=0\;\; \mbox{on}~ \partial \mathcal{B}_{R}(p)
\end{equation*}
is given by
$$u_{0}(r)=(N\kappa)^{-\frac{1}{N-1}}\int_{r}^{R}\Big(\int_{\mathcal{B}_{t}(p)}fdx\Big)^{\frac{1}{N-1}}\frac{dt}{t}.$$
Set $\widetilde{u}=u-\inf_{\mathcal{B}_{R}(p)}u$, there holds
\[
\left\{
\begin{array}{rll}
-Q_{N}\widetilde{u}&\!=-Q_{N}u\geq-Q_{N}u_{0},  &\quad\mbox{in}~ \mathcal{B}_{R}(p),\\
\widetilde{u}&\!\geq u_{0}, &\quad\mbox{on}~ \partial \mathcal{B}_{R}(p).
\end{array}
\right.
\]
The comparison principle (see \cite[Theorem 4.2]{FK2009}) implies that $\widetilde{u}\geq u_{0}$ in $\mathcal{B}_{R}(p).$ Therefore, for each $0<r<R$,
$$\widetilde{u}(p)\geq (N\kappa)^{-\frac{1}{N-1}}\int_{r}^{R}\Big(\int_{\mathcal{B}_{t}(p)}fdx\Big)^{\frac{1}{N-1}}\frac{dt}{t}\geq (N\kappa)^{-\frac{1}{N-1}}\Big(\int_{\mathcal{B}_{r}(p)}fdx\Big)^{\frac{1}{N-1}}\ln\frac{R}{r}.$$
We get then the first estimate in the Lemma by tending $r$ to $0$.\end{proof}

Combining Proposition \ref{pro2.1} with Lemma \ref{lem2.2}, we will apply the blow-up analysis to show the $\sup+\inf$ type inequality.
%, which plays a crucial role in the proof of our main results.
\begin{lem}\label{lem3.1}
Let $\{u_n\}$ be a sequence of weak solution to \eqref{main eq4} with
$$0 < M_1\leq V_{n} \leq M_{2}, \quad V_{n}\rightarrow V \quad\mbox{in }C(\overline{\Omega}),$$
and
$$\int_{\Omega}e^{u_{n}}dx\leq M,$$
for some positive constants $M_{1}$, $M_{2}$ and $M$. Then for any positive constant $C_{1}>N-1$ and any compact $\Sigma\subset\Omega$, there is constant $C_{2}>0$ such that
\begin{align}
\label{sup+inf}\max_{\Sigma}u_{n} +C_{1}\inf_{\Omega}u_{n} \leq C_{2}.\end{align}
\end{lem}

\begin{proof}
Let $C_1 > N-1$ and $\Sigma \subset\subset \Omega$ be given. If $\max\limits_{\Sigma}u_{n} +C_{1}\inf\limits_{\Omega}u_{n} <0$, we are done. So we assume that
\begin{align}
\label{geq0}
\max_{\Sigma}u_{n} +C_{1}\inf_{\Omega}u_{n} \geq0,
\end{align}
and we only need to show
\begin{align}\label{2.2}
\max_{\Sigma}u_{n} \leq C_{3},
\end{align}
for some positive constant $C_{3}$. Arguing by contradiction, we suppose that, up to a subsequence if necessary, $\{u_{n}\}$ solving \eqref{main eq4} with the above assumptions but
%$$-Q_{N}u_{n}=V_{n}e^{u_{n}} \;\;\mbox{in}~ \Omega\quad \mbox{and}\quad
$$\max_{\Sigma}u_{n}\rightarrow + \infty \quad\mbox{as}~ n\rightarrow\infty.$$
Let $\overline{p}_{n}\in \Sigma$ such that $u_{n}(\overline{p}_{n})=\max\limits_{\Sigma}u_{n}$ and  $\overline{\delta}_{n}:=e^{-\frac{u_{n}(\overline{p}_{n})}{N}}$, so
$$u_{n}(\overline{p}_{n})\rightarrow \infty, \quad\mbox{hence } \lim_{n\to \infty} \overline{\delta}_{n} = 0.$$
Choose $\eta>0$, such that
$$\Sigma_{\eta}:=\{x\in\R^N: \exists \; y\in\Sigma \mbox{ such that } F^{0}(x-y)\leq2\eta\}\subset \Omega.$$
For any $R>0$, there exists $n_{0}$ such that $R\overline{\delta}_{n}<\eta$ for all $n \geq n_{0}$. It follows from Proposition \ref{pro2.1} that
\begin{align*}
u_{n}(\overline{p}_{n})-\inf_{\mathcal{B}_{\eta}(\overline{p}_{n})}u_{n}(x)
\geq C\Big(\int_{\mathcal{B}_{R\overline{\delta}_{n}}(\overline{p}_{n})}V_{n}e^{u_{n}}dx\Big)^{\frac{1}{N-1}}\times \ln\frac{\eta}{R\overline{\delta}_{n}}.
\end{align*}
On the other hand, seeing \eqref{geq0}, there holds
$$-\inf_{\mathcal{B}_{\eta}(\overline{p}_{n})}u_{n}(x)\leq\frac{1}{C_{1}}u_{n}(\overline{p}_{n}),$$
so that
$$\Big(\int_{\mathcal{B}_{R\overline{\delta}_{n}}(\overline{p}_{n})}V_{n}e^{u_{n}}dx\Big)^{\frac{1}{N-1}}\leq \frac{1}{C}\Big(1+\frac{1}{C_{1}}\Big)\frac{u_{n}(\overline{p}_{n})}{\ln\frac{\eta}{R\overline{\delta}_{n}}}.$$
Consequently,
$$\limsup_{n\rightarrow\infty}\Big(\int_{\mathcal{B}_{R\overline{\delta}_{n}}(\overline{p}_{n})}V_{n}e^{u_{n}}dx\Big)^{\frac{1}{N-1}}\leq \frac{N}{C}\left(1+\frac{1}{C_{1}}\right), \quad \forall\; R > 0.$$

Assume $\overline{p}_{n}\rightarrow\overline{x}_{0}\in\Sigma$. Define $\overline{u}_{n}(x):=u_{n}(\overline{\delta}_{n}x+\overline{p}_{n})+N\ln\overline{\delta}_{n}$, we have
$$-Q_{N}\overline{u}_{n}(x) =V_{n}(\overline{\delta}_{n}x+\overline{p}_{n})e^{\overline{u}_{n}(x)} \:\; \mbox{in }  \frac{\Omega-{\{\overline{p}_{n}}\}}{\overline{\delta}_{n}}; \quad \overline{u}_{n}(x) \leq \overline{u}_{n}(0) = 0\;\; \mbox{in } \frac{\Sigma-{\{\overline{p}_{n}}\}}{\overline{\delta}_{n}}$$
and for any $R > 0$,
$$
\limsup\limits_{n\rightarrow\infty}\int_{\mathcal{B}_{R}}e^{\overline{u}_{n}}dx \leq \frac{1}{M_1}\left(\frac{N}{C}\right)^{N-1}\left(1+\frac{1}{C_{1}}\right)^{N-1} =:M'.$$
Theorem B implies that the following alternative holds:
\begin{itemize}
\item[(1)] $\overline{u}_{n}$ is bounded in $L_{loc}^{\infty}(\R^N)$;
\item[(2)] $V_{n}(\overline{\delta}_{n}x+\overline{p}_{n})e^{\overline{u}_{n}}\rightharpoonup \sum_{1\le i \le m}\alpha_{i}\delta_{a_{i}}$ weakly in the sense of measures over $\mathbb{R}^{N}$, and $u_{n}\rightarrow-\infty$ uniformly on compact subset of $\mathbb{R}^{N}\backslash\{a_{1},\cdots,a_m\}$.
\end{itemize}

If the case (1) appears, the elliptic estimates (see \cite{D1983}) imply that up to a subsequence $\overline{u}_{n}\rightarrow \overline{u}$ in $C^1_{loc}(\mathbb{R}^N)$ as $n\rightarrow\infty$ and $\overline{u}$ satisfies $\overline u(0)=0$,
$$
-Q_{N}\overline{u} = V(\overline{x}_{0})e^{\overline{u}} \;\; \mbox{in} ~\mathbb{R}^{N},\quad
\int_{\mathbb{R}^{N}}e^{\overline{u}}dx \leq M'.$$
Applying the classification result in \cite{CL2023}, $\overline{u}$ satisfies \eqref{bubble1} for $\lambda >0$ and $p \in \R^N$.
We see that $\overline{u}$ is radially decreasing with respect to $p$. Hence, for any $R> F^0(p) + 1$, there is a sequence $z_{n}\rightarrow p$ such that
$$\max_{\overline{\mathcal{B}}_{2R}(z_{n})}\overline{u}_{n}(x) = \overline{u}_{n}(z_{n}) \rightarrow \overline{u}(p).$$
Setting $p_{n}:=\overline{\delta}_{n}z_{n}+\overline{p}_{n}$ and $\delta_{n}=e^{-\frac{u_{n}(p_{n})}{N}}$, there holds
$$u_{n}(p_{n})=u_{n}(\overline{\delta}_{n}z_{n}+\overline{p}_{n})=\overline{u}_{n}(z_{n})-N\ln\overline{\delta}_{n}\geq \overline{u}_{n}(0) - N\ln\overline{\delta}_{n}=u_{n}(\overline{p}_{n}),$$
and
$$1\leq \frac{\overline{\delta}_{n}}{\delta_{n}}=e^{\frac{u_{n}(p_n)-u_{n}(\overline{p}_{n})}{N}}=e^{\frac{\overline{u}_{n}(z_{n})}{N}}\rightarrow e^{\frac{\overline{u}(p)}{N}}.$$

Let us now rescale $u_n$ with respect to $p_n$ by setting $\widetilde{u}_{n}(x)=u_{n}(\delta_{n}x+p_n)+N\ln\delta_{n}$, then we have $\widetilde{u}_{n}(0)=0$ and
\begin{align*}
\widetilde{u}_{n}(0) =\overline{u}_{n}(z_{n})-N\ln\overline{\delta}_{n}+N\ln\delta_{n} &=\max_{\mathcal{B}_{2R}(z_{n})}\overline{u}_{n}(x)-N\ln\overline{\delta}_{n}+N\ln\delta_{n}\\
&=\max_{\mathcal{B}_{2R}(z_n)}u_{n}(\overline{\delta}_{n}(x-z_n)+ p_n)+N\ln\delta_{n}\\
&=\max_{\mathcal{B}_{2R\frac{\overline{\delta}_{n}}{\delta_{n}}}}u_{n}(\delta_{n}x+p_n)+N\ln\delta_{n}\\
&=\max_{\mathcal{B}_{2R\frac{\overline{\delta}_{n}}{\delta_{n}}}}\widetilde{u}_{n}(x).
\end{align*}
Therefore, $\widetilde{u}_{n}(x)$ satisfies $\widetilde{u}_{n}(x)\leq \widetilde{u}_{n}(0) =0$ in $ \mathcal{B}_{2R\frac{\overline{\delta}_{n}}{\delta_{n}}}$,
\begin{align}
\label{tilde-un}
-Q_{N}\widetilde{u}_{n} =V_{n}(\delta_{n}x+p_n)e^{\widetilde{u}_{n}}\;\; \mbox{in }~ \mathcal{B}_{2R\frac{\overline{\delta}_{n}}{\delta_{n}}},\quad
\limsup\limits_{n\rightarrow\infty}\int_{\mathcal{B}_{R\frac{\overline{\delta}_{n}}{\delta_{n}}}}e^{\widetilde{u}_n(x)}dx \leq M'.
\end{align}
It follows from the standard estimates that up to a subsequence, $\widetilde{u}_{n} \rightarrow\widetilde{u}$ in $C_{loc}^{1, \beta}(\mathbb{R}^{N})$ for $0<\beta<1$ with $\widetilde{u}(x) \le \widetilde{u}(0)$, hence $\widetilde{u}$ satisfies \eqref{bubble1} with now $p =0$ and $\lambda > 0$, in particular
\begin{align}\label{blow}
\int_{\mathbb{R}^N} V(\overline{x}_{0}) e^{\widetilde{u}(x)}dx = \cc_N\kappa = N\kappa\Big(\frac{N^2}{N-1}\Big)^{N-1}.
\end{align}
Meanwhile, we have that for any $R>0$ and $\varepsilon \in (0, 1)$, there exists $n_1$ such that for any $n \geq n_1$, there hold $\mathcal{B}_{R\delta_{n}}(p_n)\subset \mathcal{B}_{\eta}(p_n)\subset \mathcal{B}_{2\eta}(\overline{p}_{n}) \subset \Omega$ and
$$V_{n}\geq V(\overline{x}_{0})\sqrt{1-\varepsilon}\quad\mbox{and}\quad u_{n}\geq w_n +\ln\sqrt{1-\varepsilon} \quad\mbox{in}~\mathcal{B}_{R\delta_{n}}(p_n)$$
where $w_n(x) = \widetilde u\Big(\frac{x - p_n}{\delta_n}\Big) - N\ln\delta_n.$
Let $f_{n} =(1-\varepsilon)V(\overline{x}_{0})e^{w_n}\chi_{\mathcal{B}_{R\delta_{n}}(p_n)}$, we have $$V_{n}e^{u_{n}}\geq f_{n} \quad\mbox{in}\quad \mathcal{B}_{\eta}(p_n).$$
It follows from Lemma \ref{lem2.2} that
$$u_{n}(p_n)-\inf_{\mathcal{B}_{\eta}(p_n)}u_{n}(x)\geq \Big[\frac{1-\varepsilon}{N\kappa}\int_{\mathcal{B}_{R}}V(\overline{x}_{0})e^{\widetilde{u}}dx\Big]^{\frac{1}{N-1}}\times\ln\frac{\eta}{R\delta_{n}},$$
that is,
$$\Big[\frac{1-\varepsilon}{N\kappa}\int_{\mathcal{B}_{R}}V(\overline{x}_{0})e^{\widetilde{u}}dx\Big]^{\frac{1}{N-1}}\leq \frac{u_{n}(p_n)-\inf_{\mathcal{B}_{\eta}(p_n)}u_{n}}{\ln\frac{\eta}{R\delta_{n}}}.$$
Recalling that $u_n(p_n) \geq u_n(\overline p_n) = \max_\Sigma u_n$, there holds
$$u_{n}(p_n) + C_1\inf_{\mathcal{B}_{\eta}(p_n)}u_n  \geq \max_\Sigma u_n + C_1\inf_{\Omega}u_{n} \geq 0.$$
We get
$$ \frac{u_{n}(p_n)-\inf_{\mathcal{B}_{\eta}(p_n)}u_{n}}{\ln\frac{\eta}{R\delta_{n}}} \leq \Big(1+\frac{1}{C_{1}}\Big)\frac{u_{n}(p_n)}{\ln\frac{\eta}{R\delta_{n}}} = \Big(1+\frac{1}{C_{1}}\Big)\frac{u_{n}(p_n)}{\ln\frac{\eta}{R} + \frac{u_{n}(p_n)}{N}}.
$$
Therefore, passing $n \to \infty$, for any $R>0$ and $\varepsilon \in (0, 1)$,
$$\Big[\frac{1-\varepsilon}{N\kappa}\int_{\mathcal{B}_{R}}V(\overline{x}_{0})e^{\widetilde{u}}dx\Big]^{\frac{1}{N-1}}\leq N\Big(1+\frac{1}{C_{1}}\Big).$$
Letting $R\rightarrow\infty$ and $\varepsilon\rightarrow 0$, we observe a contradiction with \eqref{blow} since $C_1>N-1$.

\medskip
Now we suppose that the case (2) occurs. Without loss of generality, we assume that $a_{1}=0$ and there exists $R_{1}>0$ such that $\overline{\mathcal{B}}_{2R_{1}}\cap\{a_{2},\cdots,a_m\}=\emptyset$. Since
$$\max_{\overline{\mathcal{B}}_{2R_{1}}}\overline{u}_{n}\rightarrow\infty \quad \mbox{and}\quad u_{n}\rightarrow-\infty \quad\mbox{locally uniformly in } \overline{\mathcal{B}}_{2R_{1}}\backslash\{0\},$$
there exists $q_{n}\rightarrow0$ such that
$$\overline{u}_{n}(q_{n})=\max_{\overline{\mathcal{B}}_{2R_{1}}}\overline{u}_{n}.$$
Setting now $p_n=\overline{\delta}_{n}q_{n}+\overline{p}_{n}$ and $\delta_{n}=e^{-\frac{u_{n}(p_n)}{N}}$, we have
$$u_{n}(p_n)=u_{n}(\overline{\delta}_{n}q_{n}+\overline{p}_{n})=\overline{u}_{n}(q_{n})-N\ln\overline{\delta}_{n}\geq \overline{u}_{n}(0) -N\ln\overline{\delta}_{n}=u_{n}(\overline{p}_{n}),$$
and
$$\frac{\overline{\delta}_{n}}{\delta_{n}}=e^{\frac{u_{n}(p_n)-u_{n}(\overline{p}_{n})}{N}}=e^{\frac{\overline{u}_{n}(q_{n})}{N}}\rightarrow\infty \quad\mbox{as }n\rightarrow\infty.$$
Define $\widetilde{u}_{n}(x)=u_{n}(\delta_{n}x+p_n)+N\ln\delta_{n}$. Similarly as above, we see that $\widetilde{u}_{n}(x)\leq \widetilde{u}_{n}(0) =0$ in $ \mathcal{B}_{R_1\overline{\delta}_{n}/\delta_{n}}$ and \eqref{tilde-un} is valid with $R_1$ instead of $R$. Then we can proceed exactly as in case (1) to reach again a contradiction, so we omit the details.

\medskip
To conclude, the argument by contradiction means that the sup+inf type estimate \eqref{sup+inf} is valid with any $C_1 > N-1$.
\end{proof}

\medskip
Considering $\widetilde{u}(x)=u(rx)+N\ln r$ over $\mathcal{B}_{1}$, we get readily
\begin{cor}\label{cor3.1}
Under the same hypotheses of Lemma \ref{lem3.1} with $\Omega$ replaced by $\mathcal{B}_r$ ($r > 0$), we have
$$u+C_{1}\inf_{\mathcal{B}_{r}}u(x)+N(C_{1}+1)\ln r\leq C_{2}.$$
\end{cor}

Next, we establish a Harnack type inequality, where we will use the classical local $L^{\infty}$ estimate due to Serrin, see \cite[Theorem 6]{S1964}.
\begin{prop}\label{pro2.2}
Let $u\in W^{1,N}_{loc}(\Omega)$ be a non-negative weak solution of
$$-Q_{N}u=f \quad\mbox{in }\Omega,$$
where $f\in L^{\frac{N}{N-\varepsilon}}(\Omega)$, for some $0<\varepsilon\leq1$. Then for any $\mathcal{B}_{2R}\subset\Omega$, there is a positive constant $C_0 > 1$ depending on $N$, $\varepsilon$, $\Omega$ and $F$, such that
%\footnote{\color{red} Here the constant should depend also on $F$, right?},
$$
\sup_{\mathcal{B}_{R}}u\leq C_0\left(\min\limits_{\mathcal{B}_{R}}u+\| f\|^{\frac{1}{N-1}}_{L^{\frac{N}{N-\varepsilon}}(\Omega)}\right).
$$
\end{prop}
\begin{lem}\label{lem3.3}
For $R>0$, $0<R_{0}<\frac{R}{4}$. Let $u$ be a weak solution of
$$-Q_{N}u =Ve^{u} \;\; \mbox{ and  }\;\; u(x)+N\ln F^{0}(x) \leq M  \quad \mbox{in}~\A_{R_0, R}$$
where $\A_{r, R} :=\{x\in\mathbb{R}^{N}: r<F^{0}(x)<R\}$ means the Wulff annulus, and $\|V\|_{L^{\infty}(\A_{R_0, R})}\leq M$ for some $M > 0$. Then there exist constants $C$ and $\alpha\in(0,1)$ depending only on $M$, $N$ and $F$, such that
$$
\sup\limits_{\p\B_r}u \leq \alpha\inf\limits_{\p\B_r}u +N(\alpha-1)\ln r+C,\quad\mbox{for all } r\in \Big[2R_{0},~\frac{R}{2}\Big].
$$
\end{lem}
\begin{proof}
For $r\in[2R_{0}, \frac{R}{2}]$, let $\widetilde{u}(x)=u(rx)+N\ln r$, we have
$$
-Q_{N}\widetilde{u}(x)=V(rx)e^{\widetilde{u}(x)}\;\; \mbox{and}\;\; \widetilde{u}(x) \leq M + N\ln2 \quad\mbox{in } \A_{\frac{1}{2}, 2}.
$$
 Let $g(x)= M+N\ln2-\widetilde{u}(x)\geq0$, so
$$
-Q_{N}g =Q_{N}\widetilde{u} =-V(rx)e^{\widetilde{u}}=:f\in L^{\infty}\big(\A_{\frac{1}{2}, 2}\big).
$$
It follows from Proposition \ref{pro2.2} that
$$
\sup\limits_{\p\B_1}g(x)\leq C_0\Big[\inf\limits_{\p\B_1}g(x)+\| f\|^{\frac{1}{N-1}}_{L^{\frac{N}{N-\varepsilon}}(\A_{\frac{1}{2}, 2})}\Big] \leq C_0 \inf\limits_{\p\B_1}g(x) + C.
$$
Hence
$$\sup\limits_{\p\B_1}\widetilde{u}(x)\leq\frac{1}{C_0}\inf\limits_{\p\B_1}\widetilde{u}(x)+\Big(1-\frac{1}{C_0}\Big)(M+N\ln2)+ \frac{C}{C_0}.$$
By definition of $\widetilde u$, we obtain easily the estimate on $\p\B_r$ for $u$ with $\alpha=\frac{1}{C_0}\in (0,1).$
\end{proof}

\section{Quantification of blow-up: Proof of Theorem \ref{thm1}}\label{sec4}

In order to show the quantization result, we need to clarify the situation near each blow-up point. Without loss of generality, we may assume ${\{0}\}$ is a blow-up point, and there exists $R>0$ such that $\mathcal{B}_{R}$ does not contain any other blow-up point. So we consider $\{u_{n}\}$, a sequence of weak solutions to $-Q_{N}u_{n}=V_{n}e^{u_{n}}$ in $\mathcal{B}_{R}\subset\mathbb{R}^{N}$, where
\begin{align}\label{l4.2}
0\leq V_{n}\rightarrow V \quad\mbox{in}~C(\overline{\mathcal{B}}_{R});
\end{align}
\begin{align}\label{l4.3}
\max_{\overline{\mathcal{B}}_{R}}u_{n}\rightarrow\infty, \quad \mbox{and}\quad
\max_{\overline{\mathcal{B}}_{R}\backslash \mathcal{B}_{r}}u_{n}\rightarrow-\infty, \;\;\mbox{for any}~ r\in(0,R).
\end{align}
Moreover we assume that there is $C >0$ such that
\begin{align}\label{l4.6}
\int_{\mathcal{B}_{R}}e^{u_{n}}dx\leq C.
\end{align}
To describe the blow-up behavior of $\{u_{n}\}$ near the origin, we will prove the following proposition.% Then, Theorem \ref{thm1} follows from Lemma \ref{lem4.3} and Theorem \ref{thm2} follows from Lemma \ref{lem4.4} immediately.
\begin{prop}\label{pro3.1}
Let $\{u_{n}\}$ be as above, then up to a subsequence, we have $m$ sequences $\{p_{n, j}\}_{1\le j \le m}$ in $\mathcal{B}_{R}$ and $m$ sequences of $\{k_{n, j}\}_{1\le j \le m} \in (0, \infty)$ with $\lim\limits_{n\rightarrow\infty}p_{n, j}=0$, $\lim\limits_{n\rightarrow\infty}k_{n, j}=\infty$, such that
\begin{align}\label{l4.2.1}
u_{n}(p_{n, j})=\max_{\overline{\mathcal B}_{2k_{n, j}\delta_{n, j}}(p_{n, j})}u_{n}\rightarrow\infty, \quad\forall~ 1\leq j\leq m,
\end{align}
where $\delta_{n, j}=e^{-\frac{u_{n}(p_{n, j})}{N}}$,
\begin{align}\label{l4.2.2}
\mathcal{B}_{2k_{n,i}\delta_{n,i}}(p_{n,i})\cap \mathcal{B}_{2k_{n, j}\delta_{n, j}}(p_{n, j})=\emptyset, \quad\forall ~1\leq i \ne j\leq m,
\end{align}
\begin{align}\label{l4.2.3}
\frac{\partial}{\partial t}u_{n}(ty+p_{n, j})|_{t=1}<0, \quad \forall ~\delta_{n, j}\leq F^{0}(y)\leq 2k_{n, j}\delta_{n, j},\; 1\leq j\leq m,
\end{align}
\begin{align}\label{l4.2.4}
\lim_{n\rightarrow\infty}\int_{\mathcal{B}_{2k_{n, j}\delta_{n, j}}(p_{n, j})}V_{n}e^{u_{n}}dx=\lim_{n\rightarrow\infty}\int_{\mathcal{B}_{k_{n, j}\delta_{n, j}}(p_{n, j})}V_{n}e^{u_{n}}dx=\mathcal{C}_{N}\kappa, \quad 1\leq j\leq m,
\end{align}
\begin{align}\label{l4.2.5}
u_{n}(x)+N\ln\Big(\min_{1\leq j\leq m}F^{0}(x-p_{n, j})\Big)\leq C' < \infty, \quad \forall ~n \in \N,\; x \in \mathcal{B}_R.
\end{align}
Then $V(0) > 0$, and in the sense of measure,
 $$V_{n}e^{u_{n}}\rightharpoonup m\mathcal{C}_{N}\kappa\delta_{0}, \;\; \mbox{i.e. } \lim_{n\to\infty}\int_{\mathcal{B}_{R}}V_{n}e^{u_{n}}dx = m\mathcal{C}_{N}\kappa.$$
 %for some positive integer $m$.
%$$V>0 \quad\mbox{and}\quad \beta\geq C_{N}\kappa.$$
\end{prop}

\begin{proof}
We will divide the proof into five steps. In the following, many claims hold up to a subsequence and we are interested in the asymptotic behavior of sequences of blow-up solutions. To simplify the writing, we will not repeat always ``{\it up to subsequence}'' or ``{\it for $n$ large enough}'' and we denote the subsequence always by $u_n$. Seeing \eqref{l4.6}, we assume that
$$\beta=\lim_{n\rightarrow\infty}\int_{\mathcal{B}_{R}}V_{n}e^{u_{n}}dx.$$

\textbf{Step 1. $V(0)>0$ and $\beta\geq\mathcal{C}_{N}\kappa$.}

Let $p_n\in \mathcal{B}_{R}$ satisfy $u_{n}(p_n)=\max\limits_{\overline{\mathcal{B}}_{R}}u_{n}(x)$, by (\ref{l4.3}), we have $p_n\rightarrow0$ and $u_{n}(p_n)\rightarrow\infty$. Hence $\delta_{n}=e^{-\frac{u_{n}(p_n)}{N}}$ tends to $0$.

For $F^{0}(x)\leq \frac{R}{2\delta_n}$, consider $\widetilde{u}_{n}(x)=u_{n}(\delta_{n}x+p_n)+N\ln\delta_{n}$ for $n$ large, then
$$
-Q_{N}\widetilde{u}_{n} = V_{n}(\delta_{n}x+p_n)e^{\widetilde{u}_{n}}, \; \;
\widetilde{u}_{n}\leq \widetilde{u}_{n}(0) = 0 \quad \mbox{in }  \mathcal{B}_{\frac{R}{2\delta_n}}$$
and
$$\int_{\mathcal{B}_{\frac{R}{2\delta_n}}}e^{\widetilde{u}_{n}}dx\leq C, \quad \forall \; n.$$
By Theorem B, only the case $(i)$ may occur for $\widetilde u_n$. Applying the elliptic estimates \cite{D1983}, up to a diagonal procedure,  $\{\widetilde{u}_{n}\}$ converges to $\widetilde{u}$ in $C^{1, \alpha}_{loc}(\mathbb{R}^{N})$ ($0<\alpha<1$), which satisfies $e^{\widetilde{u}} \in L^1(\R^N)$, and
$$ -Q_{N}\widetilde{u} =V(0)e^{\widetilde{u}},  \;\;  \widetilde{u} \leq \widetilde{u}(0) = 0  \quad \mbox{in }\mathbb{R}^{N}.$$

If $V(0)=0$, then $-Q_{N}\widetilde{u}_{n} = 0$ in $\mathbb{R}^{N}$. As $\widetilde{u}\leq \widetilde{u}(0) = 0$, the Liouville theorem of Finsler $N$-harmonic function (see \cite{HKM1993}) implies that $\widetilde{u} \equiv 0$, which contradicts with $e^{\widetilde{u}} \in L^1(\R^N)$, so $V(0)>0$.

Thanks to the classification result in \cite{CL2023}, there are $p \in \R^N$ and $\lambda > 0$ satisfying
\begin{align}\label{class}
\widetilde{u}(x)=\ln\frac{\mathcal{C}_{N}\lambda^{N}}{V(0)\left[1+\lambda^{\frac{N}{N-1}}\left(F^{0}(x-p)\right)^{\frac{N}{N-1}}\right]^{N}}.
\end{align}
Recall that $\widetilde{u} \leq \widetilde{u}(0)=0$ in $\mathbb{R}^{N}$, we have
$$p =0,\quad \lambda=\left(\frac{V(0)}{\mathcal{C}_{N}}\right)^{\frac{1}{N}} \quad\mbox{and}\quad V(0)\int_{\mathbb{R}^{N}}e^{\widetilde{u}(x)}dx=\mathcal{C}_{N}\kappa.$$
Given every $r>0$, we have
$$\beta=\lim\limits_{n\rightarrow\infty}\int_{\mathcal{B}_{R}}V_{n}e^{u_{n}}dx\geq\lim\limits_{n\rightarrow\infty}\int_{\mathcal{B}_{r\delta_{n}}(p_n)}V_{n}e^{u_{n}}dx=V(0)\int_{\mathcal{B}_{r}}e^{\widetilde{u}}dx.$$
Therefore, let $r\rightarrow\infty$, there holds $\beta\geq \mathcal{C}_{N}\kappa$.

\medskip
\textbf{Step 2. Quantification of single bubble case: $m=1$}

By Step 1, $V(0) > 0$. Using \eqref{l4.2}, we have $R_1 \in (0, R)$ such that
\begin{align}\label{l4.4.2}
0 < a\leq V_{n}(x)\leq b, \quad\forall ~x\in \overline{\mathcal{B}}_{R_1},\; n~\mbox{large enough}.
\end{align}
Let $p_n, \delta_n$ be that in Step 1, we may assume $p_n \in \mathcal{B}_{R_1}$ and $u_{n}(p_n)=\max_{\overline{\mathcal{B}}_{R_1}}u_{n}$.
Suppose now
\begin{align}
\label{3.8new}
u(x) + N\ln F^0(x - p_n) \leq C < \infty, \quad \forall\; n \in \N, x\in \overline{\mathcal{B}}_{R_1}.
\end{align}
We will show that $\beta = {\mathcal C}_N\kappa$.

Let $\widetilde{u}_{n}$ and $\widetilde u$ be that in Step 1. As $\widetilde{u}_{n} \rightarrow \widetilde u$ in $C^1_{loc}(\R^N)$, up to a diagonal process, there exists a sequence $k_n \to +\infty$ such that
\begin{align}
\label{3.7new}
\lim_{n\rightarrow\infty}\int_{\mathcal{B}_{k_{n}\delta_{n}}(p_{n})}V_{n}e^{u_{n}}dx = \lim_{n\rightarrow\infty}\int_{\mathcal{B}_{k_{n}}}V_{n}(\delta_n x + p_n)e^{\widetilde u_{n}}dx = \mathcal{C}_{N}\kappa .
\end{align}
Set $r_{n} =k_{n}\delta_{n}$. Up to a subsequence, recalling that $p_n \to 0$, we have
\begin{itemize}
\item[(1)] either there is $d > 0$ satisfying $r_n - F^0(p_n) \geq d$ for $n$ large enough;
\item[(2)] or $\lim_{n\to\infty} r_n = 0$.
\end{itemize}

In case (1), using (\ref{l4.3}), we have
$$\int_{\mathcal{B}_R\backslash {\mathcal B}_{r_n}(p_n)}V_{n}e^{u_{n}}dx \leq b\int_{\mathcal{A}_{d, R}}e^{u_{n}}dx\rightarrow 0 \quad\mbox{as}~ n\rightarrow\infty,$$
hence
$$\lim_{n\to\infty}\int_{\mathcal{B}_{R}}V_{n}e^{u_{n}}dx = \lim_{n\to\infty}\int_{\mathcal{B}_{r_{n}}(p_n)}V_{n}e^{u_{n}}dx = \mathcal{C}_{N}\kappa.$$

Suppose now we are in case (2). Thanks to \eqref{3.8new}, we can apply Lemma \ref{lem3.3} with $u_n(y+p_n)$ for $n$ large, and get
\begin{align}\label{l4.3.1}
\sup\limits_{\partial \mathcal{B}_r}u(x)\leq\alpha\inf\limits_{\partial \mathcal{B}_r}u(x)+N(\alpha-1)\ln r+C, \quad\forall~0 < r\leq\frac{R_{1}}{4}.
\end{align}
Here $r = F^0(x-p_n)$. Combining \eqref{l4.3.1} with Corollary \ref{cor3.1}, there holds
$$\sup\limits_{\partial \mathcal{B}_{r}}u_{n} \leq C -\frac{\alpha}{C_{1}}u_{n}(p_n)-N\left(1+\frac{\alpha}{C_{1}}\right)\ln r, \quad\forall~ 0 < r\leq\frac{R_{1}}{4}.$$
Hence
$$e^{u_{n}(x)}\leq C\delta_{n}^{\frac{N\alpha}{C_{1}}} F^{0}(x-p_n)^{-N\left(\frac{\alpha}{C_{1}}+1\right)},\quad\forall~ 0 < r\leq\frac{R_{1}}{4}.$$
We obtain then
\begin{align}\label{l4.3.3}
\int_{\mathcal{B}_{\frac{R_{1}}{4}}\backslash \mathcal{B}_{r_{n}}(p_n)}V_{n}e^{u_{n}}dx\leq C\delta_{n}^{\frac{N\alpha}{C_{1}}}\int_{r_{n}}^{\infty}r^{-\frac{N\alpha}{C_{1}}-1}dr=C k_n^{-\frac{N\alpha}{C_{1}}}\longrightarrow 0 \quad\mbox{as } n\rightarrow\infty.
\end{align}
By (\ref{l4.3}), (\ref{3.7new}) and (\ref{l4.3.3}), we obtain that $\beta = \mathcal{C}_{N}\kappa.$ with \eqref{3.8new}.

Moreover, we notice that \eqref{3.8new} is indeed equivalent to \eqref{l4.2.5} with $m = 1$ and $p_{n, 1} = p_n$. So the proof of Proposition \ref{pro3.1} when $m=1$ is completed since \eqref{l4.2.1}-\eqref{l4.2.4} are readily valid.

\medskip
\textbf{Step 3. Equivalence between \eqref{th6} and $m=1$.}

Assume that \eqref{th6} holds true. Without loss of generality, let $a_1 = 0$. Let $p_n$ be that in Steps 1-2, we need only to check \eqref{l4.2.5} with $m=1$ or equivalently \eqref{3.8new}. Using (\ref{th6}), there holds
\begin{align}\label{l4.4.3}
F^{0}(p_n)\leq e^{\frac{C_{1}-u_n(p_n)}{N}} = e^{\frac{C_{1}}{N}}\delta_{n}.
\end{align}
\begin{itemize}
\item[--] Let first $F^0(x - p_n) \le 2e^\frac{C_1}{N}\delta_n$, then
\begin{align*}
u_n(x) + N\ln F^0(x - p_n) &\leq u_n(x) + N\ln \delta_n + C_1 + N\ln 2\\
 & = u_n(x) - u_n(p_n) + C_1 + N\ln 2 \\
& \le C_1+ N\ln 2.
\end{align*}
\item[--] Let now $F^0(x - p_n) \ge 2e^\frac{C_1}{N}\delta_n$, by \eqref{l4.4.3}
$$F^0(x) \ge F^0(x - p_n) - F^0(p_n) \ge F^0(x - p_n) - e^\frac{C_1}{N}\delta_n \ge \frac{F^0(x - p_n)}{2},$$
therefore
$$u_n(x) + N\ln F^0(x - p_n) \le u_n(x) + N\ln F^0(x) + \ln 2 \le M'+\ln 2.$$
\end{itemize}
So \eqref{3.8new} holds true, and $m = 1$ by Step 2.

Very similarly, we can also prove that \eqref{3.8new} implies \eqref{th6}, we omit the details.

\medskip
\textbf{Step 4. Catch of multi-bubbles.}

We begin still with $p_n$, $\delta_n$ and $\widetilde u_n$ in Step 1. As $\widetilde{u}_{n}$ converges in $C^{1}_{loc}(\mathbb{R}^{N})$ to $\widetilde{u}$ given by \eqref{class}, we can select $k_{n}\rightarrow\infty$, such that $\| \widetilde{u}_{n}^{}-\widetilde{u}\|_{C^{1}(\overline{\mathcal{B}}_{2k_{n}})}\rightarrow0,$ and
$$\lim_{n\to\infty}\int_{\mathcal{B}_{2k_{n}\delta_{n}}(p_n)}V_{n}e^{u_{n}}dx =\lim_{n\to\infty}\int_{\mathcal{B}_{k_{n}\delta_{n}}(p_n)}V_{n}e^{u_{n}}dx = \mathcal{C}_{N}\kappa.$$
Moreover, as $\widetilde{u}$ is decreasing with respect to $F^0(x)$, we may claim that
\begin{align*}
%\label{l4.2.6}
\frac{\partial}{\partial t}u_{n}(ty+p_n)|_{t=1}<0,\quad\forall~\delta_{n}\leq F^{0}(y)\leq2k_{n}\delta_{n}, \; n\in \N.
\end{align*}
Clearly $p_{n, 1} = p_n$, $\delta_{n, 1} = \delta_n$ and $k_{n, 1} = k_n$ satisfy (\ref{l4.2.1}), (\ref{l4.2.3}) and (\ref{l4.2.4}) with $m=1$.

Now we will catch eventually other bubbles by induction. Suppose that we have already $\ell-1$ sequences $\{p_{n, j}\}_{1\le j \le \ell-1}$, $\{k_{n, j}\}_{1\le j \le \ell-1}$ with $\ell\geq 2$ satisfying (\ref{l4.2.1})-(\ref{l4.2.4}) with $m=\ell-1$. If \eqref{l4.2.5} holds with $m=\ell-1$, we stop the process. Otherwise, let
\begin{align*}
%\label{l4.2.7}
M_{n}: = u_{n}(\overline{p}_{n, \ell})+N\ln G(\overline{p}_{n, \ell}) = \max_{x\in\overline{ \mathcal{B}}_{R}}\big[u_{n}(x)+N\ln G(x)\big]
\end{align*}
where
\begin{align*}
G(x) := \min_{1 \leq j \leq \ell-1}F^{0}(x-p_{n, j}).
\end{align*}
Then $M_{n}\rightarrow\infty$, in particular $u_{n}(\overline{p}_{n, \ell})\rightarrow\infty$. Denoting $\overline{\delta}_{n, \ell}=e^{-\frac{u_{n}(\overline{p}_{n, \ell})}{N}}$,  we have
$$\xi_n := \frac{G(\overline{p}_{n, \ell})}{\overline{\delta}_{n, \ell}}\rightarrow\infty.$$
By triangle inequality,
 \begin{align*}
%\label{l4.2.8}
G(\overline{p}_{n, \ell}+\overline{\delta}_{n, \ell}x)\geq G(\overline{p}_{n, \ell}) -\overline{\delta}_{n, \ell}F^{0}(x) \geq \frac{G(\overline{p}_{n, \ell})}{2}, \quad \forall\; F^{0}(x)\leq\frac{\xi_n}{2}.
 \end{align*}
 Set $\overline{u}_{n}(x)=u_{n}(\overline{p}_{n, \ell}+\overline{\delta}_{n, \ell}x)+N\ln\overline{\delta}_{n, \ell}$. If $F^{0}(x)\leq \frac{\xi_n}{2}$, there holds
 \begin{align*}
 \overline{u}_{n}(x) &=u_{n}(\overline{p}_{n, \ell}+\overline{\delta}_{n, \ell}x)+N\ln G(\overline{p}_{n, \ell}+\overline{\delta}_{n, \ell}x) + N\ln\overline{\delta}_{n, \ell} - N\ln G(\overline{p}_{n, \ell}+\overline{\delta}_{n, \ell}x)\\
 &\leq u_{n}(\overline{p}_{n, \ell})+N\ln G(\overline{p}_{n, \ell})+ N\ln\overline{\delta}_{n, \ell} - N\ln G(\overline{p}_{n, \ell}+\overline{\delta}_{n, \ell}x)\\
 &= N\ln\frac{G(\overline{p}_{n, \ell})}{G(\overline{p}_{n, \ell}+\overline{\delta}_{n, \ell}x)}\\
 &\leq N\ln2.
 \end{align*}
So $\overline{u}_{n}(0) =0$, and
$$-Q_{N}\overline{u}_{n} = V_{n}(\bar{\delta}_{n, \ell}x+\bar{p}_{n, \ell})e^{\overline{u}_{n}},  \;\;
\overline{u}_{n} \leq N\ln2,  \quad   \mbox{in }\;  {\mathcal B}_\frac{\xi_n}{2}.$$
Then up to subsequence, $\overline{u}_{n}\rightarrow\overline{u}$ in $C_{loc}^{1, \alpha}(\mathbb{R}^{N})$, where
$$\overline{u}(x) =\ln\frac{\mu^{N}}{\left(1+\gamma\mu^{\frac{N}{N-1}}F^{0}(x-\overline{p})^{\frac{N}{N-1}}\right)^{N}}$$
with $\gamma ^{N-1} = \frac{V(0)}{\mathcal{C}_{N}}$, $\mu>0$ and $\overline{p}\in \mathbb{R}^{N}$.
We can choose a sequence $k_{n, \ell}\rightarrow\infty$ such that
\begin{align*}
\|\overline{u}_{n}-\overline{u}\|_{C^{1, \beta}(\overline{\mathcal{B}}_{2k_{n, \ell}})}\rightarrow 0.
\end{align*}
Therefore for $n$ large enough
\begin{align}
\label{l4.2.10}
\frac{\partial}{\partial t}\overline{u}_{n}(ty+\overline{p})|_{t=1}<0, \quad 1 \leq F^{0}(y)\leq 2k_{n}^{l}.
\end{align}

Let $y_{n, \ell}\in \mathcal{B}_1$ satisfy
\begin{align*}
%\label{l4.2.11}
\overline{u}_{n}(y_{n, \ell}+\overline{p})=\max\limits_{\overline{\mathcal{B}}_{4k_{n, \ell}}}\overline{u}_{n}(y+\overline{p}),
\end{align*}
and we choose $p_{n, \ell}=\overline{p}_{n, \ell}+\overline{\delta}_{n, \ell}(y_{n, \ell}+\overline{p})$. By the definition of $y_{n,\ell}$ and the expression of $\overline u$, there holds $y_{n, \ell}\rightarrow0$ and
\begin{align}\label{l4.2.12}
u_{n}(\overline{p}_{n, \ell})\leq u_{n}(p_{n, \ell}) \leq u_{n}(\overline{p}_{n, \ell}) + C,
\end{align}
because $\lim_{n\to\infty}\big[u_{n}(\overline{p}_{n, \ell}) - u_{n}(p_{n, \ell})] = \overline u(\overline p)-\overline u(0)$. Denote
$$\delta_{n, \ell}=e^{-\frac{u_{n}(p_{n, \ell})}{N}}, \quad \widetilde u_{n, \ell}(x) = u_n(\delta_{n, \ell}x + p_{n, \ell}) + N\ln\delta_{n, \ell}.$$
By \eqref{l4.2.12}, there holds
\begin{align*}
%\label{l4.2.13}
\delta_{n, \ell}\leq\overline{\delta}_{n, \ell}\leq C\delta_{n, \ell}.
\end{align*}
Up to a subsequence, we assume that
$$\lim_{n\to\infty} \frac{\delta_{n, \ell}}{\overline \delta_{n, \ell}} = \mu \in (0, 1]$$
Readily $\widetilde u_{n, \ell} \to \overline u(\mu x + \overline p) + N\ln\mu$ in $C^1_{loc}(\R^N)$. It is easy to see that (\ref{l4.2.1})-(\ref{l4.2.4}) are satisfied for $\{p_{n, j}\}_{1\leq j \le \ell}$ and $\{k_{n, j}\}_{1\le j \le \ell}$ with $m= \ell$. For example, \eqref{l4.2.2} is valid seeing the induction hypothesis and \eqref{l4.2.10}.

We continue this process until (\ref{l4.2.5}) holds true. In fact we will stop after a finite number of iterations since each time we add a mass of $\mathcal{C}_N\kappa$ near the blow-up point $p_{n, j} \to 0$, it follows then $m\leq \frac{C V(0)}{\mathcal{C}_N \kappa}$.

\medskip
\textbf{Step 5. No neck contribution.}

The last step is to show that there is no mass contribution out of the $m$ bubbles provided by Step 4. Here we will use the following lemma  which arrives at a slightly more general situation.
\begin{lem}\label{lem3.2}
Let $\{V_n\}$  be a sequence of functions satisfying \eqref{l4.2}. Assume that $\{u_{n}\}$ is a sequence of solutions of \eqref{main eq1} satisfying \eqref{l4.3}, \eqref{l4.6} and \eqref{l4.2.1} on $\mathcal{B}_{R}$. Let $\{p_{n, j}\}_{1\le j \le m}$ be $m$ sequences in $\mathcal{B}_{R}$ and $\{k_{n, j}\}_{1\le j \le m} \in (0, \infty)$ be $m$ sequences of positive numbers with $\lim\limits_{n\rightarrow\infty}p_{n, j}=0$, $\lim\limits_{n\rightarrow\infty}k_{n, j}=\infty$, which satisfy \eqref{l4.2.2} and
\begin{align}\label{l4.2.4bis}
\lim_{n\rightarrow\infty}\int_{\mathcal{B}_{2k_{n, j}\delta_{n, j}}(p_{n, j})}V_{n}e^{u_{n}}dx=\lim_{n\rightarrow\infty}\int_{\mathcal{B}_{k_{n, j}\delta_{n, j}}(p_{n, j})}V_{n}e^{u_{n}}dx=\beta_j, \quad 1\leq j\leq m,
\end{align}
\begin{align}\label{l4.2.5bis}
u_{n}(x)+N\ln\Big[\min_{1\leq j\leq m}F^{0}(x-p_{n, j})\Big]\leq C' < \infty, \quad \forall ~n \in \N,\; x \in \mathcal{B}_R.
\end{align}
Then
 $$ \lim_{n\to\infty}\int_{\mathcal{B}_{R}}V_{n}e^{u_{n}}dx = \sum_{j=1}^{m}\beta_j.$$
\end{lem}

\noindent
{\sl Proof of Lemma \ref{lem3.2}.} Recall that $V(0)>0$ and \eqref{l4.4.2} holds true. We will proceed by induction. By analysis as in Step 2, we can show that the result holds true for $m=1$. Since the arguments are very similar, so we omit the details.

Assume that the result holds for $m \leq \ell-1$ with $\ell \ge 2$, and we consider now $m =\ell$. Without loss of generality, assume that $p_{n, 1} = 0$ for all $n$ and
\begin{align}
\label{dn}
d_{n}= \min_{1\leq i\ne j\leq \ell} F^{0}(p_{n,i}-p_{n, j}) = F^{0}(p_{n,1}-p_{n, 2}).
\end{align}

{\sl Case 1}. Suppose first that there is $A > 0$ satisfying
\begin{align}\label{l4.3.4}
F^{0}(p_{n,i}-p_{n, j})\leq Ad_{n}, \quad 1\leq i,j\leq \ell, \; n \in \N.
\end{align}
We claim then
\begin{align}\label{l4.3.5}
\lim\limits_{n\rightarrow\infty}\int_{\mathcal{B}_{4Ad_{n}}}V_{n}e^{u_{n}}dx=\lim\limits_{n\rightarrow\infty}\int_{\mathcal{B}_{2Ad_{n}}}V_{n}e^{u_{n}}dx=\sum_{j=1}^m \beta_j.
\end{align}
Once (\ref{l4.3.5}) is established, we can apply the similar method in Step 2 for the case $m=1$ to conclude that $\beta = \sum_{1\le j \le m} \beta_j.$

Thus, it remains to prove (\ref{l4.3.5}). Set $r_{n,j}=k_{n, j}\delta_{n, j}$, consider
$$\widetilde{u}_{n}(x)=u_{n}(d_{n}x)+N\ln d_{n} \;\; \mbox{and} \;\; \widetilde{V}_{n}(x)=V_{n}(d_{n}x), \quad \mbox{for } F^{0}(x)\leq\frac{R}{d_{n}}.$$
Denote also $\widetilde{x}_{n, j}=\frac{p_{n, j}}{d_{n}}$, $\widetilde{\delta}_{n, j}=e^{-\frac{\widetilde{u}_{n}(\widetilde{x}_{n, j})}{N}}=\frac{\delta_{n, j}}{d_{n}}$ and $\widetilde{r}_{n, j}=\frac{r_{n, j}}{d_{n}}$ for $1\leq j\leq \ell$. Using \eqref{l4.2.1}, \eqref{l4.2.2} and \eqref{l4.2.5bis}, we have that $\forall\; 1\le j \le \ell$, $\lim_{n\to\infty} \widetilde{x}_{n,j}=0,$
$\lim_{n\to\infty} \widetilde{r}_{n, j}/\widetilde{\delta}_{n, j} =\infty$;
$$\mathcal{B}_{\widetilde{r}_{n,i}}(\widetilde{x}_{n,i})\cap \mathcal{B}_{\widetilde{r}_{n, j}}(\widetilde{x}_{n, j})=\emptyset,\quad\forall~ 1\leq i\ne j\leq \ell, \; n \in \N;$$
$$\widetilde{u}_{n}(x)+N\ln\Big[\min_{1\leq j\leq \ell}F^{0}(x-\widetilde{x}_{n, j})\Big] \leq C,\quad\forall ~ x\in\overline{\mathcal{B}}_\frac{R}{d_{n}}, \; n\in \N;$$
$$\lim\limits_{n\rightarrow\infty}\int_{\mathcal{B}_{2\widetilde{r}_{n, j}}(\widetilde{x}_{n, j})}\widetilde{V}_{n}e^{\widetilde{u}_{n}}dx=\lim\limits_{n\rightarrow\infty}
\int_{\mathcal{B}_{\widetilde{r}_{n, j}}(\widetilde{x}_{n, j})}\widetilde{V}_{n}e^{\widetilde{u}_{n}}dx=\beta_j, \quad\forall~ 1\leq j\leq \ell.$$
Furthermore, \eqref{l4.2.2} guarantees that $u_n(p_{n, j}) + N\ln F^0(p_{n,i}-p_{n,j}) \to \infty$ for any $1\leq i\ne j\leq \ell$, so $\widetilde u_n(\widetilde x_{n, j})$ tends to $\infty$ by \eqref{l4.3.4}.

Since $F^{0}(\widetilde{x}_{n, j})\leq A$ for all $n \in \N$ and $1\leq j\leq \ell$, up to a subsequence, there holds
\begin{align*}
%\label{l4.3.6}
\lim_{n\to\infty} \widetilde{x}_{n, j} = q_j, \quad\mbox{for }\; 1\leq j\leq \ell.
\end{align*}
Applying again Theorem B, we see that
\begin{align}\label{l4.3.7}
\widetilde{u}_{n}\rightarrow-\infty \mbox{ uniformly on compact subsets of } \mathbb{R}^{N}\backslash \{q_j, 1 \le j \le \ell\}.
\end{align}
Obviously,
\begin{align*}
%\label{l4.3.8}
1\leq F^{0}(q_i - q_j)\leq A, \quad\forall~ 1\leq i\ne j\leq \ell.
\end{align*}

For each $ 1\leq j\leq \ell$, either $\liminf_{n\rightarrow\infty}\widetilde{r}_{n,j}>0$, we get by (\ref{l4.3.7})
\begin{align}\label{l4.3.9}
\int_{\mathcal{B}_{\frac{1}{2}}(q_j)}\widetilde{V}_{n}e^{\widetilde{u}_{n}}dx\rightarrow \beta_j;
\end{align}
or $\lim_{n\rightarrow\infty}\widetilde{r}_{n, j}=0$, then (\ref{l4.3.9}) still holds by similar estimate in Step 2. Finally we conclude \eqref{l4.3.5} using (\ref{l4.3.7}) and (\ref{l4.3.9}).

\medskip
{\sl Case 2}. Suppose now \eqref{l4.3.4} is not valid, then (up to a subsequence)
$$J = \Big\{1\le j \le \ell, \lim_{n\rightarrow\infty}\frac{F^{0}(p_{n,j})}{d_{n}}=\infty\Big\} \ne \emptyset.$$
For $j \in J' = \{1,\cdots,\ell\} \backslash J$, there is $A > 0$ such that $F^{0}(p_{n,j})\leq Ad_{n}$. Without loss of generality we assume $J' =\{1,\cdots,s\}$ with $1\leq s\leq \ell-1$. Notice that $s \ge 2$ by the definition \eqref{dn}.

Denote still $\widetilde{u}_{n}(x)=u_{n}(d_{n}x)+N\ln d_{n}$ for $F^{0}(x)\leq4A$. Similar arguments as for {\sl Case 1} show that
\begin{align}\label{l4.3.10}
\lim\limits_{n\rightarrow\infty}\int_{\mathcal{B}_{4Ad_{n}}}V_{n}e^{u_{n}}dx=\lim\limits_{n\rightarrow\infty}\int_{\mathcal{B}_{2Ad_{n}}}V_{n}e^{u_{n}}dx=\sum_{j=1}^s \beta_j.
\end{align}

Set $r_n' =Ad_{n}$ and $x_n'=0$, we see that $(\ell-s+1)$ sequences $\{x_n'\}\cup \{p_{n, j}\}_{s < j \leq \ell}$ with radii $r_n', r_{n, j} (s < j \le \ell)$ and the masses $\beta'=\sum_{1\leq j \leq s} \beta_j$, $\beta_j (s < j \le \ell)$ satisfy
(\ref{l4.2.1}), (\ref{l4.2.4bis})-(\ref{l4.2.5bis}) with $m' := \ell-s+1 < \ell$. We only need to show
\begin{align}\label{l4.3.12}
\mathcal{B}_{Ad_{n}}\cap \mathcal{B}_{r_{n, j}}(p_{n, j})=\emptyset, \quad\forall ~s < j\leq \ell.
\end{align}
Suppose that (\ref{l4.3.12}) does not hold for some $j > s$. Then $Ad_{n}+r_{n, j}\geq F^{0}(p_{n, j})$, it follows that
$$\frac{r_{n, j}}{d_{n}}\geq \frac{F^{0}(p_{n, j})}{d_{n}}-A\rightarrow\infty, \quad\mbox{as } n\rightarrow\infty.$$
We have, for $n$ large enough,
$$Ad_{n}+F^{0}(p_{n, j})\leq 2r_{n, j} \quad \mbox{since } r_{n, j}\geq \frac{2}{3}F^{0}(p_{n, j}).$$
Therefore $\mathcal{B}_{Ad_{n}}\subset \mathcal{B}_{2r_{n, j}}(p_{n, j})$, which is a contradiction with (\ref{l4.2.2}).

Now we can apply the induction hypothesis with $m' < \ell$ to claim $\beta = \sum_{1 \leq j \leq \ell} \beta_j$. Lemma \ref{lem3.2} is established. \qed

\medskip
Proposition \ref{pro3.1} is clearly valid using Lemma \ref{lem3.2} with $\beta_j=\mathcal{C}_N \kappa.$
\end{proof}

The conclusions of Theorem \ref{thm1} follow readily from Proposition \ref{pro3.1}. So we are done.

\bigskip
\noindent{\bf Acknowledgements.}   X. Huang is partially supported by NSFC (No.~12271164). Y. Li is partially supported by NSFC (No.~12401132), and also partially supported by China Postdoctoral Science Foundation (No.~2022M721164). F. Zhou is partially supported by NSFC (No.~12071189). All authors are also supported in part by Science and Technology Commission of Shanghai Municipality (No.~22DZ2229014).

\bigskip

\noindent{\bf Data availability.}  Data sharing not applicable to this article as no datasets were generated or analysed during the current study.

\end{document}